\documentclass[letterpaper, 10 pt, conference]{ieeeconf}
\IEEEoverridecommandlockouts

\usepackage{cite}
\usepackage{amsmath,amssymb,amsfonts}
\usepackage{graphicx}
\usepackage{textcomp}
\usepackage{xcolor}
\usepackage{hyperref}
\usepackage{mathtools}
\usepackage{float}
\usepackage{cite}
\usepackage{xfrac}
\usepackage[euler]{textgreek}

\usepackage{algpseudocode,algorithm,algorithmicx}

\algrenewcommand\algorithmicrequire{\textbf{Precondition:}}
\algrenewcommand\algorithmicensure{\textbf{Postcondition:}}

\newtheorem{prop}{Proposition}
\newtheorem{defn}{Definition}

\newcommand*{\V}{\ensuremath{\mathcal{V}}}

\newcommand*{\E}{\ensuremath{\mathcal{E}}}

\newcommand*{\Hstar}{\ensuremath{\mathcal{H}}^*}
\newcommand*{\RHstar}{\ensuremath{\mathcal{RH}}^*}
\newcommand*{\PRHstar}{\ensuremath{\mathcal{PRH}}^*}

\newcommand*{\X}{\ensuremath{\mathcal{X}}}

\newcommand*{\C}{\ensuremath{\mathcal{C}}}
\newcommand*{\N}{\ensuremath{\mathcal{N}}}
\newcommand*{\T}{\ensuremath{\mathcal{T}}}

\DeclareMathOperator*{\argmin}{arg\,min}

\def\BibTeX{{\rm B\kern-.05em{\sc i\kern-.025em b}\kern-.08em T\kern-.1667em\lower.7ex\hbox{E}\kern-.125emX}}

\begin{document}
\addtolength{\abovedisplayskip}{-.05cm}
\addtolength{\belowdisplayskip}{-.05cm}
\addtolength{\textfloatsep}{-.5cm}

\title{\LARGE \bf Efficient Path Planning with Soft Homology Constraints}

\author{Carlos A. Taveras, Santiago Segarra, and C\'esar A.~Uribe
\thanks{ 
Research was sponsored by the Army Research Office and was accomplished under Grant Number W911NF-17-S-0002. The views and conclusions contained in this document are those of the authors and should not be interpreted as representing the official policies, either expressed or implied, of the Army Research Office or the U.S. Army or the U.S. Government. The U.S. Government is authorized to reproduce and distribute reprints for Government purposes notwithstanding any copyright notation herein.
Part of this material is based upon work supported by the National Science Foundation under Grants \#2211815 and \#2213568 and the Ken Kennedy-HPE Cray 2023/24 Graduate Fellowship.
The authors are with the Dept. of ECE, Rice University. 
Emails:\texttt{\{ctaveras, segarra, cauribe\}@rice.edu}. 
}}
\maketitle

\begin{abstract}
We study the problem of path planning with soft homology constraints on a surface topologically equivalent to a disk with punctures. Specifically, we propose an algorithm, named $\Hstar$, for the efficient computation of a path homologous to a user-provided reference path. 
We show that the algorithm can generate a suite of paths in distinct homology classes, from the overall shortest path to the shortest path homologous to the reference path, ordered both by path length and similarity to the reference path. 
Rollout is shown to improve the results produced by the algorithm.
Experiments demonstrate that $\Hstar$ can be an efficient alternative to optimal methods, especially for configuration spaces with many obstacles. \label{sec:abstract}
\end{abstract}

\section{Introduction} \label{sec:1}
Motion planning algorithms are ubiquitous in robotics,
serving as the bridge between high-level task specifications and the low-level instructions that completion of the task necessitates \cite{LaValle_2006}.
Whether due to the intrinsic properties of a robot or the presence of obstacles in an environment, most real-world configuration spaces in motion planning problems have non-trivial topology \cite{Choset_Lynch_Hutchinson_Kantor_Burgard_2005}.
Paths in a configuration space with non-trivial topology can be partitioned into equivalence classes by topological equivalence relations, namely homotopy and homology.
Topology-constrained path planning has proven to be useful in several robotics applications, including multi-agent coordination \cite{Wu_Bhattacharya_Prorok_2020}, motion planning in dynamic environments \cite{Kolur_Chintalapudi_Boots_Mukadam_2019}, and guided AUV navigation \cite{Hernandez_Carreras_Ridao_2011}.

As previously mentioned, homotopy and homology can define equivalence relations on paths. 
Two paths with identical start and end points are homotopic if they can be continuously deformed into one another. They are said to be homologous if the concatenation of the paths does not enclose any holes.
Homotopy and homology classes are defined as the set of all homotopic and homologous paths, respectively.
The Hurewicz theorem relates these equivalences \cite{hatcher}, implying that homotopic paths are homologous (the converse, however, is false; see \cite[Figure 2]{Bhattacharya_Likhachev_Kumar_2012}).

The fields of computational geometry and topology have long studied topology constrained path planning problems over discrete surfaces.
An important line of work in these fields is that of finding the shortest path between a pair of points belonging to a designated homotopy class.
For example, \cite{HERSHBERGER199463} makes use of the funnel algorithm \cite{chazelle, Lee1984EuclideanSP} to solve this problem for boundary-triangulated 2-manifolds.
Recent work has extended these results to a more general class of combinatorial surfaces \cite{tightening}.
Related work includes solving for shortest homologous cycles \cite{Chambers_Erickson_Nayyeri_2009, optimal_homologous_cycles_lp} and 
computing a minimal homology basis \cite{Dey_Li_Wang_Minimal, Chen_Freedman_2010}. 
In  recent years, the robotics motion planning community has developed various search \cite{Bhattacharya_Likhachev_Kumar_2012, Bhattacharya_Lipsky_Ghrist_Kumar_2013, Demyen, Kuderer_Sprunk_Kretzschmar_Burgard_2014}, sampling \cite{Yi_Goodrich_Seppi_2016, Schmitzberger_Bouchet_Dufaut_Wolf_Husson_2002, Sandstrom_Uwacu_Denny_Amato_2020}, and optimization  
\cite{Kolur_Chintalapudi_Boots_Mukadam_2019} based topology aware motion planning methods.

Despite these developments, however, there is no efficient method for producing shortest paths in homology classes that are similar to a designated class.
Such a method would enable users to produce paths that are similar to the shortest path in the designated class, but with shorter length.
Towards this end, \cite{Bhattacharya_Likhachev_Kumar_2012} introduces the $H$-signature, a homology class invariant, that can be used to create an augmented graph (or road map) over which computing the shortest path in a homology class is trivial.
This method can be inefficient, however, as it often requires the computation of paths belonging to homology classes that are very different from the designated class and of little practical relevance (e.g., paths that loop around holes).

Our method enables users to tailor their search to paths in homology classes similar to the designated class efficiently, as experiments validate.
Other works related to ours include 
\cite{Kolur_Chintalapudi_Boots_Mukadam_2019} and \cite{Kuderer_Sprunk_Kretzschmar_Burgard_2014} which can efficiently construct paths in distinct homotopy classes, however, they do not explicitly aim to minimize the length of the resulting paths as our method and \cite{Bhattacharya_Likhachev_Kumar_2012} do. 

We propose $\Hstar$ an efficient search-based algorithm that can produce a sequence of paths in distinct homology classes for configuration spaces homeomorphic to a disk with holes, ordered by length and the degree of proximity to a user-designated homology class.
We propose a distance metric between homology classes and show that paths that are close in this metric appear topologically similar.
As in \cite{Yi_Goodrich_Seppi_2016}, we only consider paths that do not form complete loops around obstacles/holes.

This paper is organized as follows.
Section \ref{sec:2} establishes the necessary background, notation, and preliminary results.  
In Section \ref{sec:3}, we state the objectives of this paper. In Section \ref{sec:4}, we propose and justify algorithms for solving the objectives. 
In Section \ref{sec:5}, we demonstrate the effectiveness of the proposed algorithms and compare them with the methods proposed in \cite{Bhattacharya_Likhachev_Kumar_2012}.
We conclude in Section \ref{sec:6} with a summary.

\section{Preliminaries} \label{sec:2}

\subsection{Simplicial Complexes}
\label{sec:2A}

\begin{figure*}
    \centering
    \includegraphics[height=4 cm, width= 16.5 cm]{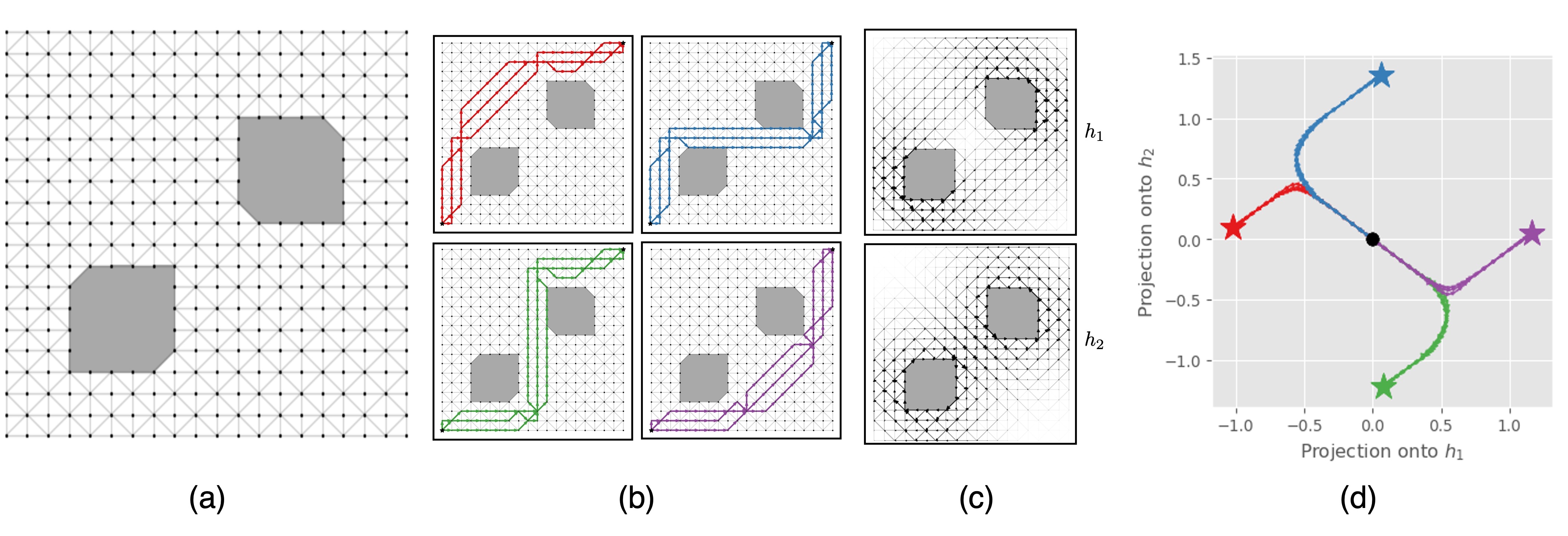}
    \caption{(a) Simplicial Complex with two holes, (b) four distinct homotopy/homology classes with three paths each, (c) harmonic projection basis vectors, and (d) the progression of the harmonic projection for each path in (b).}
    \label{fig:example}
    \vspace{-4 mm}
\end{figure*}
Let $\V = \{v_1, ..., v_N\}$  be a finite set of elements called vertices.
We call a non-empty subset $\sigma \subset \V$ with $|\sigma| = k+1$ a $k$-simplex, and say it has dimension $k$.  
An (abstract) simplicial complex (SC) $\X$ is a collection of simplices that is closed under restriction \cite{ghrist2014elementary}; i.e., if $\sigma \in \X$ and $\pi \subseteq \sigma$ then $\pi \in \X$. 
The SCs in this paper are assumed to be discretized (manifold) surfaces \cite{crane} and are thus comprised of $\{0, 1, 2\}$-simplices which we call nodes, edges and triangles, respectively.
Denoting the sets of edges and triangles respectively as $\E, \text{ and } \T$, we can express an SC as $\X = (\V, \E, \T)$.
Each edge $e \in \E$ is assigned a weight by a function $w$ such that $w(e) > 0$. 

We assign an orientation to all nodes, edges and triangles.
Nodes have positive orientation by convention.
The canonical positive orientation of an edge/triangle is that in which all nodes are in label order (e.g., $[v_1, v_2, v_3]$).
A pair of oriented edges/triangles are equivalent if they differ by an even permutation (e.g., $[v_1, v_2, v_3] = [v_3, v_1, v_2]$).
An oriented edge/triangle can be negated by any odd permutation of its elements (e.g., $-[v_1, v_2] = [v_2, v_1]$). 
Orientation and sign can be extended to simplices of any dimension; we denote the set of oriented $k$-simplices by $\X^{k}$.

The $k$-th chain group $\C_k$ of $\X$ is the vector space over $\mathbb{R}$ generated by a basis formed by the oriented $k$-simplices. 
As $\C_k$ is a finite-dimensional vector space, its elements, called $k$-chains, can be represented as vectors with real coefficients. 
The boundary operator $\partial_k$ is a linear mapping from $\C_k$ to $\C_{k-1}$ which acts on $\sigma \in  \X^{k}$ by
$\partial_k(\sigma) = \sum_{i=0}^{m} (-1)^{i} \sigma^{-i}$
where $\sigma^{-i} = [\sigma_0, ..., \sigma_{i-1},  \sigma_{i+1}, ..., \sigma_m]$ is the $(k-1)$-simplex formed by removing the $i$-th node of $\sigma$. 
As a linear operator $\partial_k$ can be encoded by a matrix.

Recommended texts on simplicial complexes, their applications, and related concepts include \cite{hatcher, nanda21, Edelsbrunner_Harer, ghrist2014elementary}.

\subsection{Paths} 
\label{sec:2B}
\vspace{-2 mm}
Let $v_i, v_f \in \mathcal{V}$ be distinct nodes, which we respectively call the source and destination.
We represent a path as a sequence of edge-connected nodes $\tau = (\tau^{(0)}, ..., \tau^{(n)})$ where $\tau^{(0)} = v_i$ and $\tau^{(n)} = v_f$.
We denote the (weighted) length of $\tau$ by 
\begin{equation}\label{eq:weight}
\vspace{-2 mm}
W(\tau) = \sum_{i=0}^{n-1} w(\tau^{(i)}, \tau^{(i+1)})
\end{equation}
We define simple algebraic operations on paths as follows.
Let $\tau_1=(\tau_1^{(0)},...,\tau_1^{(n)})$ and $\tau_2=(\tau_2^{(0)},...,\tau_2^{(m)})$ be paths.
Addition is defined by path concatenation $\tau_1 + \tau_2 = (\tau_1^{(0)}, ..., \tau_1^{(n)}, \tau_2^{(1)}, ..., \tau_2^{(m)})$ given that $\tau_1^{(n)} = \tau_2^{(0)}$. 
Negation is defined as path reversal $-\tau_1 = (\tau_1^{(n)}, ..., \tau_1^{(0)})$ and subtraction by $\tau_1 - \tau_2 = \tau_1 + (-\tau_2)$.
The sum of a path and a node is $\tau_1 + (v) = \tau_1 + (\tau_1^{(n)}, v)$ given that $(\tau_1^{(n)}, v) \in \mathcal{E}$.

A path from $v_i$ to $v_f$ can also be represented as a $1$-chain $x \in C_1$ satisfying $\partial_1 (x) = -v_i + v_f$.
The map $\Phi$ can transform a path $\tau$ to $x \in C_1$ 
\begin{equation}\label{eq:seqtochain}
x = \Phi(\tau)=\sum_{i=0}^{n-1} \phi(\tau^{(i)}, \tau^{(i+1)})
\end{equation}
where $\phi(v_i, v_j) = [v_i, v_j]$ if $i < j$, and $\phi(v_i, v_j) = -[v_j, v_i]$, otherwise.

\begin{defn}(Homologous paths)
Paths $\tau_1$ and $\tau_2$ connecting points $v_i, v_f$ are \textit{homologous}, denoted $\tau_1 \sim \tau_2$, if the difference of their $1$-chain representations $\Phi(\tau_1) - \Phi(\tau_2)$ is a $2$-chain boundary (i.e. $\Phi(\tau_1) - \Phi(\tau_2) \in \text{im}(\partial_2)$).
\end{defn}
It can be shown that $\sim$ is an equivalence relation on paths with fixed endpoints. 
We refer to an equivalence class of homologous paths as a homology class.

\subsection{Hodge Theory, Homology and Harmonics} \label{sec:2C}

The composition of consecutive boundary maps is null, that is, for all $x \in \C_{k+1}$, $\partial_{k} \circ \partial_{k+1}(x) = 0$~\cite{nanda21}.
As a consequence, we can define the quotient vector space $H_k = \text{ker}(\partial_k) / \text{im}(\partial_{k+1})$, called the $k$-th homology group, where $\text{im}(\cdot)$ and $\text{ker}(\cdot)$ respectively denote the image and kernel of an operator.
Moreover, the Hodge $k$-Laplacian can be defined using the boundary operator as
\begin{equation} \label{eq:laplacian}
L_k = \partial_k^{\top} \partial_k + \partial_{k+1} \partial_{k+1}^{\top}.
\end{equation}
The Hodge Decomposition Theorem \cite{Grady_Polimeni_2010} states that $C_k$ can be decomposed into three orthogonal subspaces
\begin{equation} \label{eq:hodge}    
C_k \cong \text{im}(\partial_{k+1}) \oplus \text{im}(\partial_{k}^{\top}) \oplus \text{ker}(L_k)
\end{equation}
where $\oplus$ denotes the direct sum between a pair of subspaces.

Discrete Hodge Theory implies that $H_k$ and $\text{ker}(L_k)$ are isomorphic \cite{Lim_2019}, allowing for an algebraic condition to determine whether a pair of paths are homologous.
Toward this, we construct a matrix $H = [h_1 ... h_D]$, where $D$ is the number of holes in the SC, whose columns span $\text{ker}(L_1)$ and define an operator which we call the harmonic projection of a path $\tau$
\begin{equation} \label{eq:proj}
\gamma(\tau) = H^\top \Phi(\tau).
\end{equation}

\begin{prop} \label{prop:1}
Two paths $\tau_1$ and $\tau_2$ connecting the same points are homologous if and only if $\gamma(\tau_1) = \gamma(\tau_2)$.
\end{prop} 
\begin{proof}
Let $x_1 = \Phi(\tau_1)$ and $x_2 = \Phi(\tau_2)$.

Suppose $\tau_1 \sim \tau_2$.
Therefore, $x_1 - x_2 \in \text{im}(\partial_2)$.
As $\text{ker}(L_1) = \text{im}(H)$ and $\text{im}(\partial_2) \cap \text{ker}(L_1) = \{0\}$ by Eq. \eqref{eq:hodge}, the fundamental theorem of linear algebra \cite{strang} implies $x_1 - x_2 \in \ker(H^\top)$, or $H^\top (x_1 - x_2) = 0$.
Equivalently $H^\top x_1 = \gamma(\tau_1)  = \gamma(\tau_2) =  H^\top x_2$.

Suppose now that $\gamma(\tau_1) = H^\top x_1 = H^\top x_2 = \gamma(\tau_2)$.
Hence, $x_1 - x_2 \in \text{ker}(H^\top) = \text{im}(H)^\perp$.
By construction $\text{im}(H) = \text{ker}(L_1)$ which by Eq. \eqref{eq:hodge} implies that $x_1 - x_2 \in \text{ker}(L_1)^\perp = \text{im}(\partial_2) \oplus \text{im}(\partial_1^\top)$.
As $\tau_1$ and $\tau_2$ are connected by the same points $x_1 - x_2 \in \text{ker}(\partial_1) = \text{im}(\partial_1^\top)^\perp$.
As $\text{im}(\partial_1^\top)^\perp \cap \text{im}(\partial_1^\top) = \{0\}$, we have $x_1 - x_2 \in \text{im}(\partial_2)$ thus $\tau_1 \sim \tau_2$. 
\end{proof}

Works that leverage harmonic 1-(co)chains to measure topological similarity between paths include \cite{random_walks, Frantzen_outliers, ghosh_2018, Schaub_higher_order, yin, roddenberry2019hodgenet, Roddenberry_Glaze_Segarra, jiagraph2019}.
Proposition \ref{prop:1} is analogous to \cite[Lemma 2]{Bhattacharya_Likhachev_Kumar_2012} as the holomorphic functions in $H$-signature are harmonic \cite{crane}.

\section{Problem Statement} \label{sec:3}
Let $\X = (\mathcal{V}, {\mathcal{E}}, {\mathcal{T}})$ be an oriented 2-dimensional simplicial complex topologically equivalent to a disk with $D$ holes.
Let $v_i, v_f \in \mathcal{V}$ respectively denote the source and destination nodes.
Let $\bar{\tau} = (\bar{\tau}^{(0)}, ..., \bar{\tau}^{(n)})$ be a path called the reference path with harmonic projection $\bar{\gamma} = \gamma(\bar{\tau})$.
We aim to efficiently compute a path $\tau^*$ homologous to $\bar{\tau}$ with minimal length; in other words, we want to solve
\begin{equation*}
\begin{aligned}   
\tau^* = \argmin_{\tau \in \mathcal{P}} \quad & W(\tau) \hspace{2 mm}
\textrm{s.t.} \hspace{2 mm}\gamma(\tau) = \gamma(\bar{\tau}).
\end{aligned}
\tag{P1}\label{P1}
\end{equation*}
where $\mathcal{P}$ is the set of all edge-connected paths with source $v_i$ and destination $v_f$.

A search-based algorithm to solve \ref{P1} would require the ability to compute several paths connecting $v_i$ to $v_f$ that are the shortest in their respective class.
The method proposed in \cite{Bhattacharya_Likhachev_Kumar_2012} enables this by lifting the problem to an augmentation of the original graph in which an augmented node $v' = (v, H(v))$ consists of both a node $v \in \V$ and an $H$-signature (harmonic projection) $H(v) \in \mathbb{C}^{D}$. 
A $\tau^*$ can thus be found by solving for a path connecting $v_i' = (v_i, 0)$ to $v_f' = (v_f, \bar{\gamma})$ in the augmented graph using $A^*$ \cite{astar}.

Depending on the length of $\tau^*$, however, this algorithm can require constructing extremely large graphs, which we would like to avoid, so we restrict our focus to solutions in the original problem space. 

To bypass the homology constraint, we soften it, introducing a penalty to the cost function, as one does in Lagrangian-based optimization \cite{bertsekas2014constrained}. 
Let 
\begin{equation} \label{eq:proj_diff}
\Delta \gamma(\tau_1, \tau_2) = \lVert \gamma(\tau_1) - \gamma(\tau_2) \rVert_2    
\end{equation}
denote the harmonic projection difference between $\tau_1$ and $\tau_2$, where $\lVert \cdot \rVert_2$ is the Euclidean norm.
We reformulate \ref{P1} as
\begin{equation*}
\begin{aligned}   
\min_{\tau \in \mathcal{P}} \quad & C^\alpha(\tau, \bar{\tau}).
\end{aligned}
\tag{P2}\label{P2}
\end{equation*}
where $C^\alpha(\tau, \bar{\tau}) = W(\tau) + \alpha \Delta \gamma(\tau, \bar{\tau})$ and $\alpha > 0$.
This reformulation cannot be solved optimally on the SC using dynamic programming-based (DP) methods as the cost function lacks the optimal substructure that characterizes DP problems \cite{Bertsekas_2012}.
Therefore, we develop a heuristic function to approximate a solution to \ref{P2} that leverages the sequential structure of the harmonic projection (as seen in Fig. \ref{fig:example}(d)) and use rollout to improve results produced by the heuristic.
It can be shown that for sufficiently large $\alpha$, the minimizers for \ref{P1} and \ref{P2} are equivalent.

\begin{prop} \label{prop:2}
There exists a sufficiently large $\alpha>0$ in \ref{P2}, such that a minimizer of \ref{P1} is a minimizer of \ref{P2}.
\end{prop}
\begin{proof}
Let $\tau^{*}_1$ denote the shortest path from $v_i$ to $v_f$ and $\tau^*$ be a solution to \ref{P1} (i.e. $\tau^{*}$ is shortest path such that $\tau^{*} \sim \bar{\tau}$). 
Without loss of generality, assume that there are $m \geq 1$ distinct homology classes whose respective shortest paths $\tau^*_1, ..., \tau^{*}_m$ have shorter length than $\tau^*$ (i.e. $W(\tau^{*}_1) < ... < W(\tau^{*}_m) < W(\tau^*)$). 
Proposition \ref{prop:1} states that if $\tau^* \sim \bar{\tau}$ then $\gamma(\tau^*) = \gamma(\bar{\tau})$, hence the cost of $\tau^*$ in \ref{P2} is
 $$C^\alpha(\tau^*, \bar{\tau}) = W(\tau_i^*) + \alpha \Delta \gamma(\tau^*, \bar{\tau}) = W(\tau^*) + 0 = W(\tau^*)$$ which is independent of $\alpha$.
Path $\tau^*$ is thus a minimizer of \ref{P2} when $\alpha^*$ satisfies $C^{\alpha^*}(\tau^*, \bar{\tau}) < C^{\alpha^*}(\tau^{*}_i, \bar{\tau})$ for all $\tau^{*}_i$, in particular, when
\begin{align}    
\alpha^* > \max \ \{\alpha_i\}_{i=1}^{m} = \max{} \left\{ \frac{W(\tau^*) - W(\tau^{*}_i)}{\Delta \gamma(\tau^{*}_i, \bar{\tau})} \right\}_{i=1}^{m}. \label{eq:alphas}
\end{align}
\end{proof}

We say \ref{P2} has a soft homology constraint as it penalizes paths in homology classes that are ``farther away" from the desired homology class, as measured by the projection difference $\Delta \gamma$.

\section{Methods} \label{sec:4}
\subsection{The $\mathcal{H}^*$ Algorithm} \label{sec:4A}
We introduce $\Hstar$, a homology-informed heuristic algorithm best-first search algorithm providing approximate solutions to \ref{P2}.
$\mathcal{H}^*$ is essentially $A^*$ \cite{astar} with a heuristic that measures the harmonic projection difference between $\bar{\tau}$ and a partial path from $v_i$ to a node $\tau^{(k)}$.
In particular, the cost of a node $\tau^{(k)}$ connected to $v_i$ by path $\tau_k = (\tau^{(0)}, ..., \tau^{(k)})$ is $W(\tau_k) + \alpha \Delta \gamma(\tau_k, \bar{\tau})$.
The parameter $\alpha$ thus controls the extent to which the projection difference contributes to the overall cost at a node. 
Note that when the destination node is reached, its cost coincides with the cost function in \ref{P2}.

The $\mathcal{H}^*$ heuristic was inspired by the equivalence of homologous paths in harmonic projection (Proposition \ref{prop:1}) and an observation about the similarity of sequential structure of paths within a homology class, as can be seen in Fig. \ref{fig:example}(d).
Notice, in particular, that over each edge in a path in Fig. \ref{fig:example}(d), the projection difference tends to decrease until the destination is reached.

\begin{figure*} 
\centering
\includegraphics[width=16.5 cm, height=4.5cm]{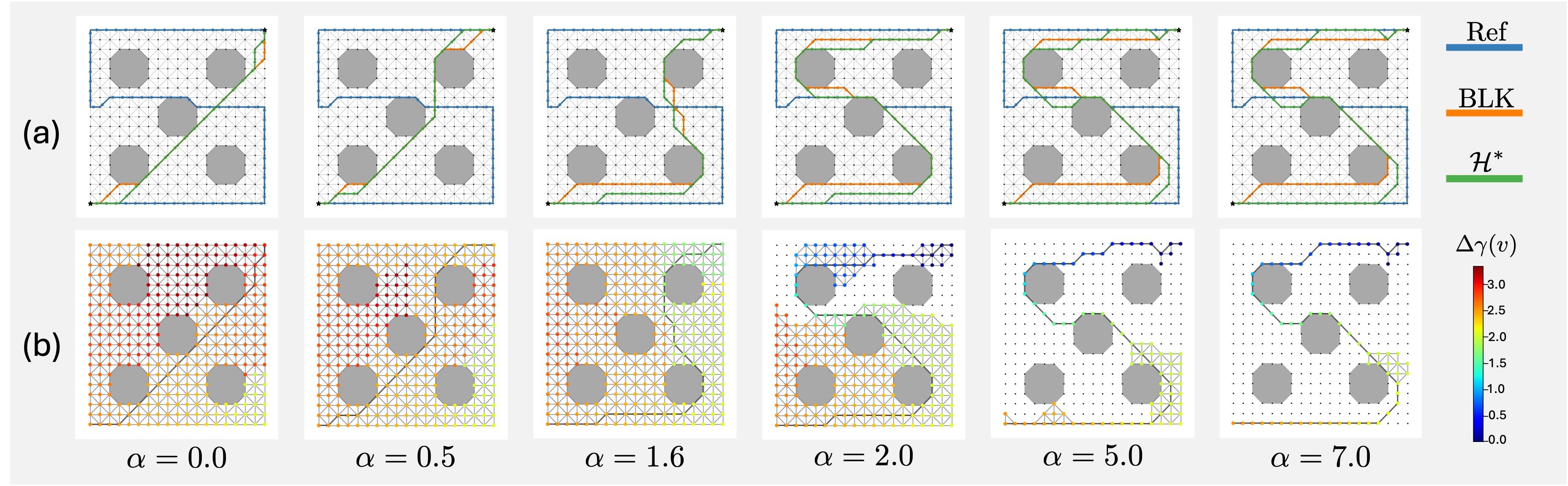}
\caption{(a) Paths produced by $\Hstar$ for various $\alpha$ and (b) the projection difference of all visited nodes after a path has been found.}
\label{fig:pedagogy}
\vspace{-5 mm}
\end{figure*}

We now describe the parameters and steps of the $\Hstar$ algorithm.
At a stage $k$, a set $A_k$ of visited nodes is maintained; we denote the set of unvisited nodes by  $A_k^c = \V \setminus A_k$.
Each node $v\in \V$ is associated with: i) a weight $W_k(v)$, ii) harmonic projection $\gamma_k(v)$, iii) cost $C^\alpha_k(v) = W_k(v) + \alpha \Delta \gamma_k(v)$ where $\Delta \gamma_k(v) = \lVert \gamma_k(v) - \bar{\gamma} \rVert_2$, and iv) the node $\text{prev(v)}$ preceding $v$ in the path from $v_i$ to $v$.

At stage $0$, we initialize the aforementioned data to $W_0(v) = \mathbb{I}(v, v_i), \gamma_0(v) = \mathbf{0}_D \in \mathbb{R}^D, C^\alpha_0(v) = W_0(v) + \alpha \Delta \gamma_0(v)$, and $\text{prev} = \emptyset$ where $\mathbb{I}(u,v) = 0$ if $u = v$, and $\infty$, otherwise, and $\emptyset$ means that a node has no predecessor.

For $k>0$, we compute the node in $A_{k-1}^c$ with least cost, 
\begin{equation}    
v_k^* = \argmin_{v \in A_{k-1}^c} 
C^\alpha_{k-1}(v),
\end{equation}
add it to the visited node set $A_k = A_{k-1} \cup \{v_k^*\}$, then update the cost of its neighbors, which we denote by $\mathcal{N}(v_k^*) = \{ u \in \V \ | \ (u,v_k^*) \in \E\}$.
In particular, the cost of a node $v \in \mathcal{N}(v_k^*) \cap A_k^c$ is updated according to 
\begin{align}
C^\alpha_k(v) & = \min \{C^\alpha_{k-1}(v), W'(v) + \alpha \Delta \gamma'(v)\} \\
W'(v) &= W_{k-1}(v_k^*) + w(v_k^*, v) \\
\Delta \gamma' (v) &= 
\lVert \gamma(v_k^*) + \gamma(v_k^*, v) -  \bar{\gamma} \rVert_2 
\end{align}
where, if the path through $v_k^*$ has lower cost, then $\text{prev}(v) = v_k^*$;
nodes outside of this set do not change at this stage.
The algorithm terminates when $v_k^* = v_f$, after which, the path from $v_i$ to $v_f$ is reconstructed using $\text{prev}$. 
We use $\Hstar(v_i, v_f, \tau)$ to denote the output of the $\Hstar$ algorithm.
\begin{figure*} 
    \centering
    \includegraphics[height=4.3 cm, width=16.5 cm]{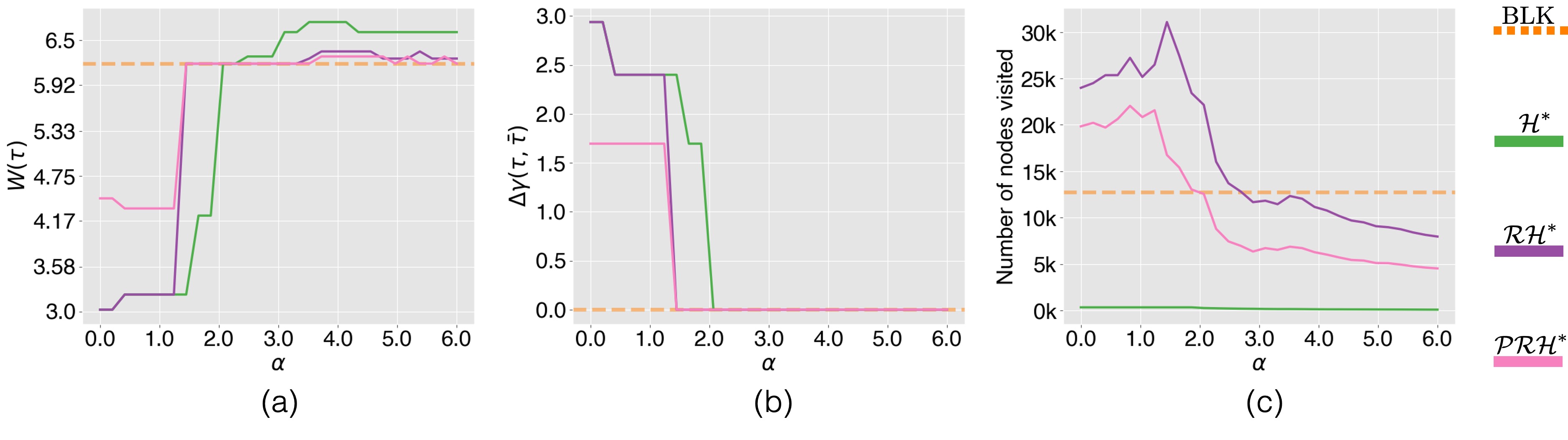} 
    \vspace{-3 mm}
    \caption{(a) Path length, (b) projection difference, and (c) number of nodes visited for $\mathcal{H}^*, \mathcal{RH}^*, \text{ and } \mathcal{PRH}^*$ as a function of $\alpha$. 
    The color-coded (and labeled) horizontal lines each correspond to a distinct homology class in (a) denoting the classes' shortest path length and (b) the projection difference. 
    The horizontal line in (c) shows the number of nodes visited by BLK.     
    \vspace{-5 mm}}
    \label{fig:alpha_sweep}
\end{figure*}

 \subsection{Fortified Rollout} \label{sec:4B}
Solutions to a combinatorial problem provided by a heuristic algorithm (not to be confused with `heuristic' in the $\text{A}^*$ sense) can be improved by embedding the heuristic in a rollout framework \cite{BertsekasRollout}.
Instead of solving for the entire path, as in \ref{P2}, rollout sequentially constructs the path using a heuristic. 
We detail how rollout is leveraged for our purposes.
At a non-terminal stage $k$, we maintain a path $\tau_k = (\tau^{(0)}, ..., \tau^{(k)})$, where $\tau^{(0)} = v_i$ and $\tau^{(k)} \neq v_f$. 
The path is updated by $\tau_{k+1} = \tau_k + (v_{k+1})$, where
\vspace{0 mm}
\begin{align}
v_{k+1} & = \argmin_{v \in \N(\tau^{(k)})}
C^{\alpha}(\tau_k + \Hstar(v,v_f,\bar{\tau}_k(v)), \bar{\tau}). \label{eq:rollout}
\vspace{-2 mm}
\end{align}
Because the source node for $\Hstar$ in Eq. \eqref{eq:rollout} changes at each stage, the reference path at stage $k$ for a node $v$ is $\bar{\tau}_k(v) = (v) -\tau_{k} + \bar{\tau}$. 

If a heuristic satisfies the sequential consistency property \cite[Definition 2.2]{BertsekasRollout}, rollout produces a path no worse than the one produces by the heuristic alone \cite[Proposition 3.2]{BertsekasRollout}.
Under mild conditions, a heuristic can be made sequentially consistent using the fortified rollout algorithm outlined in \cite[Section 4]{BertsekasRollout}.
The key difference between rollout and fortified rollout is that at every stage, the fortified variant maintains a full path $\tilde{\tau}_k$ from $v_i$ to $v_f$; 
in particular, continuing from Eq. \eqref{eq:rollout}, we have 
\begin{align}
\tilde{\tau}_{k+1} &= \argmin \{
C^{\alpha}(\tilde{\tau}_{k}, \bar{\tau}), 
C^{\alpha}(\tau'_{k+1}, \bar{\tau})
\} \\
\tau'_{k+1} &= \tau_k + \Hstar(v_{k+1}, v_f, \bar{\tau}_{k}(v_{k+1})) 
\end{align}
where $\tilde{\tau}_{0} = \Hstar(v_i, v_f, \bar{\tau})$.

While rollout can improve the performance of $\mathcal{H}^*$, it requires several uses of $\mathcal{H}^*$, thereby increasing the effective number of nodes visited to produce a result. 
We mitigate this by pruning nodes that increase the projection difference by more than some small value $\epsilon$
\begin{equation}
{\N}(v, \epsilon) = \{u \in \N(v) | \Delta \gamma(\tau_k + (u), \bar{\tau}) < \Delta \gamma(\tau_k, \bar{\tau}) + \epsilon\}.
\end{equation}
We use $\mathcal{RH}^*(v_i, v_f, \bar{\tau})$ and $\mathcal{PRH}^*(v_i, v_f, \bar{\tau}, \epsilon)$ to denote the output of fortified rollout without and with pruning, respectively.

\section{Numerical Analysis} \label{sec:5}

We show the utility of $\Hstar$ and its variants on various synthetic configuration spaces. 
Configuration spaces are constructed by triangulating a uniform grid of $19 \times 19$ of points on $[-1,1]^2  \subset \mathbb{R}^2$. 
Holes are created by removing all simplices within a specified  region and edge weights are set to the Euclidean distance between node positions. 
We construct the Hodge $1$-Laplacian $L_1$ of the configuration space and compute a basis for its null space to form $H = [h_1, ..., h_D]$.
Reference paths $\bar{\tau}$ are constructed by connecting set key points by shortest path.

For each experiment we compare the length and number of nodes visited for $\mathcal{H}^*$, $\mathcal{RH}^*$, $\mathcal{PRH}^*$ and BLK \cite{Bhattacharya_Likhachev_Kumar_2012} 
\footnote{Experiment code available here: \href{https://github.com/ctaveras1999/h-star}{https://github.com/ctaveras1999/h-star}}.
All experiments were run on a 1.1 GHz Quad-Core Intel Core i5 processor. 
Per node visit, on average, $\Hstar$, $\RHstar$ and $\PRHstar$ each take $0.001$ seconds while BLK takes $0.002$ seconds.

\subsection{Illustrative Example} \label{sec:5A}

We provide a practical example on an SC with five holes, consisting of $316$ nodes. 
Figure \ref{fig:pedagogy}(a) shows $\Hstar$ results for several $\alpha$ values (in green), given the S-shaped reference path (in blue) with source and destination nodes at the bottom-left and top-right, respectively. 
We plot the BLK path (in orange) that is homologous to the $\Hstar$ path for a given $\alpha$. 
Paths produced by $\Hstar$ for $\alpha = 0.0, 0.5, 1.6$ and $2.0$ belong to distinct homology classes, the last of which is homologous to $\bar{\tau}$.
For $\alpha \in \{0.0, 0.5, 1.6, 2.0 \}$, the $\Hstar$ and BLK paths are equal in length. 
To achieve this, $\Hstar$ visits $315, 315, 315$ and $251$ nodes, respectively, whereas BLK visits $14,269$ nodes – a two order of magnitude difference in node visits to achieve the same results.

In addition to efficiency gains, $\Hstar$ provides an interpretable interface for producing paths that are increasing similar to $\bar{\tau}$, as measured by $\Delta \gamma(\tau, \bar{\tau})$, as $\alpha$ increases. 
This ordering allows users to use $\alpha$ to trade off between path length and topological similarity to the reference path. 
The $\Hstar$ path for $\alpha=1.6$ has length $4.2$ whereas the solution to \ref{P1}, which is achieved by $\Hstar$ for $\alpha=2$, has length $6.2$. 
If the user is able to ignore the loop around the top left hole, they are provided a similar, but much shorter route, to the destination.

For $\alpha \in \{5, 7\}$, the $\Hstar$ paths are homologous to $\bar{\tau}$ both with length $6.6$, providing examples of non-optimal paths produced by $\Hstar$.
These results are achieved after visiting only $86$ and $63$ nodes, respectively, which is about a quarter of the SC's $316$ nodes and two orders of magnitude fewer than BLK. 
Such a trade-off can be justified for applications in which the speed of a planner is critical, as is the case in many real-time systems.

The larger $\alpha$ is, the earlier $\Hstar$ visits nodes with smaller projection difference. 
Due to this and the sequential structure of the harmonic projections (e.g., Fig. \ref{fig:example}(d)), $\Hstar$ builds paths to nodes that are closer in harmonic projection at the expense of optimality in path length. 
Figure \ref{fig:pedagogy}(b) demonstrates this phenomenon, showing that the harmonic projection difference decreases (or remains unchanged) at each visited node, as $\alpha$ goes from $0$ to $2.0$. 
Increasing $\alpha$ arbitrarily promotes increasingly greedy behavior towards minimizing projection difference.
This behavior can be seen in Fig. \ref{fig:pedagogy}(b) as $\alpha$ transitions from $2.5$ to $5.0$ to $7.0$.

\subsection{Characterizing the Effect of \textalpha}
In this next experiment, we aim to characterize the effect of the parameter $\alpha$ on the behavior of $\Hstar$, $\RHstar$, and $\PRHstar$.
The SC and reference path in this experiment are identical to that of the first (i.e. Fig. \ref{fig:pedagogy}(a)).
We sample thirty evenly spaced points from $0$ to $6$ for $\alpha$ and use them to produce paths with $\Hstar$ (in green), $\RHstar$ (in purple) and $\PRHstar$ (in pink). 
Figure \ref{fig:alpha_sweep} (a), (b) and (c), respectively plot the length, projection difference, and number of nodes visited by the algorithms.
The (orange) horizontal lines in Fig. \ref{fig:alpha_sweep} correspond to the output of BLK (which is independent of $\alpha$).

For $0 \leq \alpha \leq 2$, the $\Hstar$, $\RHstar$, and $\PRHstar$ results are stable (i.e. remain in the same homology class) over several ranges, eventually jumping to homology classes with smaller harmonic projection difference, at the cost of increased path length.
Moreover, for $\alpha < 2$, as shown in Fig. \ref{fig:pedagogy}(a), $\Hstar$ produces a shortest path in each homology class leading up to the desired one.
Around $\alpha = 2$, each of the algorithms produce a shortest path in the desired homology class.
Figure \ref{fig:alpha_sweep}(c) shows that, across $\alpha$, $\Hstar$ visits the least amount of nodes among all methods ($315$ at most), while achieving results comparable to BLK for each homology class $\Hstar$ produces solutions to.
From $\alpha < 2$, BLK requires the second least amount of node visits at about $19$k, followed by $\PRHstar$, then $\RHstar$ ranging between $30$k and $15$k. 
As $\alpha$ increases, the number of nodes visited by $\PRHstar$ and $\RHstar$ trends downwards, requiring less node visits than BLK for $\alpha > 2$. 
The reason for this is explained in the previous experiment and can be seen in Fig. \ref{fig:pedagogy}(b).

Importantly, $\RHstar$ and $\PRHstar$ produce shorter length paths than $\Hstar$ does for all $\alpha$ in which the algorithm outputs are homologous, as can be seen in Fig. \ref{fig:alpha_sweep}(a), demonstrating the utility of the rollout procedure. 
Finally, Fig. \ref{fig:alpha_sweep}(c) shows that node pruning, as proposed in Sec. \ref{sec:4B}, can reduce the number of nodes visited when using rollout. 

\subsection{Characterizing the Effect of the Number of Holes}
This final experiment studies how the number of holes in an SC  affects the number of nodes each algorithm visits to produce a path homologous to the reference.
We show that the number of nodes visited by BLK can scale with the number of holes, while our algorithms do not.

Towards this, we generate nine SCs, with between one and nine holes each with fixed hole location and reference path.
Examples of these SCs and reference paths can be seen in Fig. \ref{fig:more_holes}(c). 
Figure \ref{fig:more_holes}(a) and (b) respectively show the path lengths and number of node visits for each algorithm.
For $\Hstar, \RHstar$ and $\PRHstar$, the paths were produced by sampling twenty $\alpha$ values between $0$ and $3$, and choosing the path in the desired homology class with shortest length. 

In all cases, the $\Hstar$ algorithms produce optimal shortest paths while maintaining a relatively constant number of node visits for each SC.
For the first four SCs, $\Hstar$ and BLK visit a comparable amount of nodes, while $\RHstar$ and $\PRHstar$ visit significantly more.
For each hole added after the fourth, the number of nodes visited by BLK increases almost exponentially. 
This is because each time a hole is added, short paths that had previously been homologous no longer are.
Consequently, the number of homologically distinct paths between $v_i$ and $v_f$ that have shorter length than that of the desired path can increase as the number of holes increases.
For instance, in the five, six, and seven hole cases, BLK respectively produces $102, 252,$ and $675$ homologically distinct shortest paths before producing the desired path.

The $\Hstar$ algorithm, on the other hand, does not scale with the number of holes as it visits each node once, at most.
Moreover, $\RHstar$ and $\PRHstar$ scale with the number of nodes in a path and the size of the nodes' neighborhoods, which is unrelated to the number of holes.
As such, the $\Hstar$-based algorithms can be used as efficient alternatives to BLK in many-obstacled environments.

\begin{figure} 
    \centering
    \includegraphics[width=8.4 cm, height=4.5 cm]{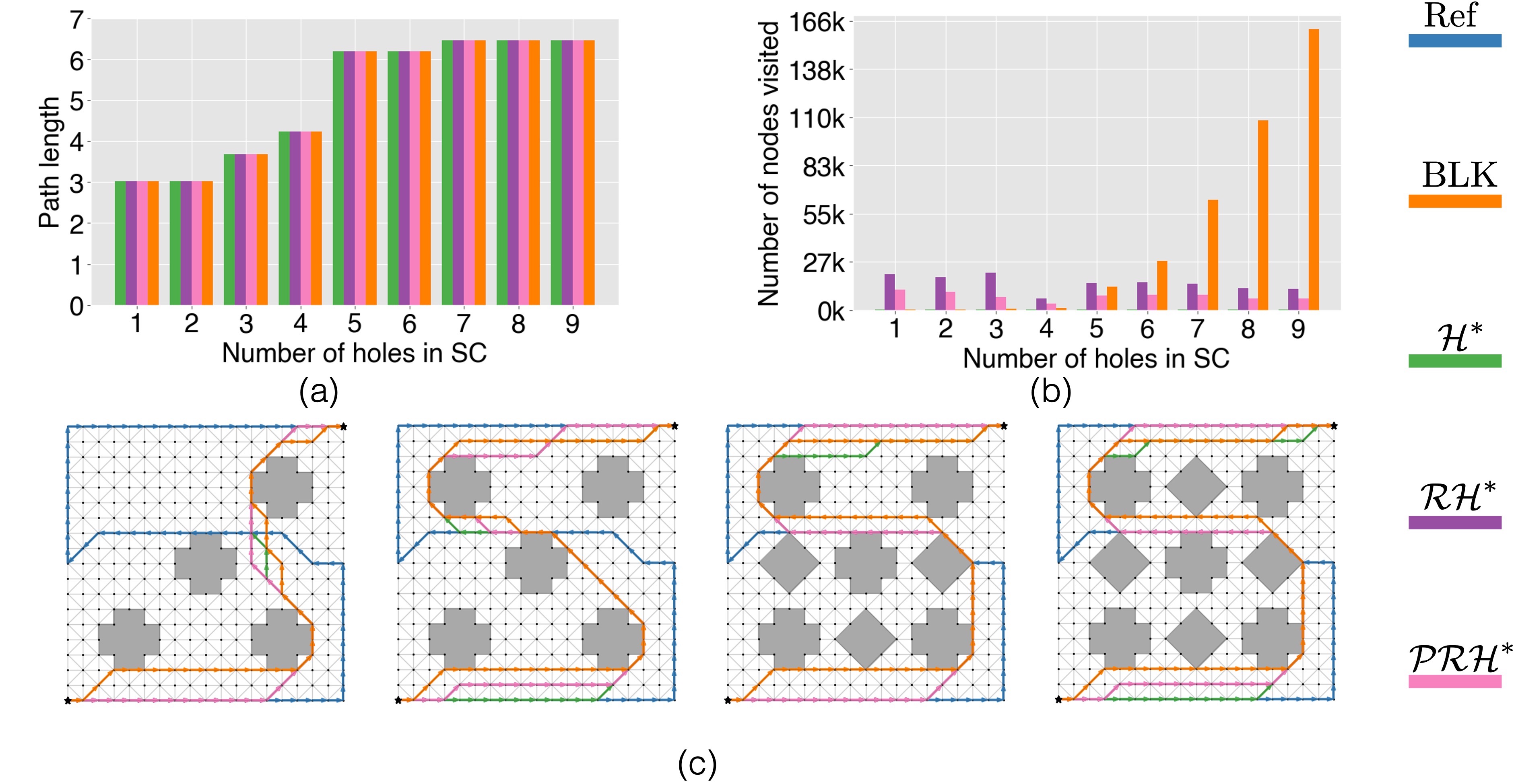}
    \vspace{-4 mm}
    \caption{(a) The length of paths produced by algorithms for each SC, (b) number of nodes visited to produce the paths, and (c) example paths.} \label{fig:more_holes}
\end{figure}

\section{Conclusion} \label{sec:6}
We present $\Hstar$, an efficient heuristic algorithm for solving a relaxation of the problem of finding the shortest path homologous to some user-provided reference path. 
We show that $\Hstar$ can be useful for suggesting paths in homology classes that are similar to, but with shorter shortest path than, the reference class. 
$\RHstar$ and $\PRHstar$ are introduced as rollout-based variants of $\Hstar$ that can improve its results.
Experiments demonstrate that $\Hstar$ can produce results that are often comparable to  \cite{Bhattacharya_Likhachev_Kumar_2012} and at a significantly reduced computational cost, especially for environments with many obstacles or holes present.

\bibliographystyle{IEEEtran}
\nocite{*}
\bibliography{refs}

\end{document}